\documentclass[reqno, 11pt]{amsart}

\usepackage{amsfonts, amsthm, amsmath, amssymb}
\usepackage{hyperref}
\hypersetup{colorlinks=false}

\usepackage{bbm}  

 \usepackage{tabu}

\usepackage[margin=1in]{geometry}

\usepackage{helvet}

\RequirePackage{mathrsfs} \let\mathcal\mathscr

\usepackage{hyperref}
\usepackage{dsfont}
\usepackage{upgreek}
\usepackage{mathabx,yfonts}
\usepackage{enumerate}

\numberwithin{equation}{section}

\newtheorem{theorem}{Theorem}[section] 
\newtheorem{lemma}[theorem]{Lemma}

\theoremstyle{definition}

\newtheorem{remark}[theorem]{Remark}

\renewcommand{\phi}{\varphi}

\renewcommand{\leq}{\leqslant}

\renewcommand{\geq}{\geqslant}

 \renewcommand{\b}{\mathbf{b}}

\DeclareSymbolFont{bbold}{U}{bbold}{m}{n}
\DeclareSymbolFontAlphabet{\mathbbold}{bbold}

\newcommand{\md}[1]{  \left(\textnormal{mod}\ #1\right)}

\renewcommand{\b}{\mathbf}

\renewcommand{\epsilon}{\varepsilon}

\renewcommand{\leq}{\leqslant}
\renewcommand{\geq}{\geqslant}
\renewcommand{\#}{\sharp}

\title[Statistics of small prime quadratic non-residues]{Statistics of small prime quadratic non-residues}

\author{Efthymios Sofos} 
\address{Department of Mathematics\\
University of Glasgow  \\ G12~8QQ United Kingdom}
\email{efthymios.sofos@glasgow.ac.uk}


\date{}
\begin{document}
\begin{abstract} 
We prove that the average of the 
$k$-th smallest prime quadratic non-residue modulo a   prime
 approximates    the $2k$-th smallest prime. 
\end{abstract}

\maketitle

\setcounter{tocdepth}{1}
\tableofcontents

\section{Introduction}  What is the expected size of the $k$-th smallest quadratic non-residue modulo $p$?
This question is interesting only when we exclude certain obvious choices. For example, 
 all small integer multiples of the first quadratic non-residue
 are quadratic non-residues.
For this reason we define   $n_k(p)$ to be the $k$-th smallest \textit{prime} 
quadratic non-residue modulo $p$. 
A well-known 
notorious problem of Vinogradov~\cite{vino} 
regards upper bounds for $n_k(p)$:
he conjectured that $n_k(p)=O_k(p^\epsilon)$ for all $\epsilon>0$. 
 For $k=1$ this is known conditionally on the validity of the Generalized Riemann Hypothesis,
see the work of Lamzouri--Li--Soundararajan~\cite{MR3356031}.
Unconditionally, the problem   is wide open.  The best result in this direction is due to Burgess~\cite{MR93504},
who   proved deep bounds for short character sums to deduce  that $n_1(p)=O(p^\epsilon)$ holds for all $\epsilon> \frac{1}{4\sqrt \mathrm e }$.
His work was extended   to all $k$ by
Banks and Guo~\cite{MR3606952}; remarkably, they achieve a bound of the same quality for all   $k\geq 1$.
Another   related  work is due to Bourgain and  Lindenstraus~\cite[Theorem 5.1]{MR1957735}, 
who produced many prime quadratic non-residues below $p$
by  exploiting  connections to Quantum Unique Ergodicity.

In this work we investigate the average value of $n_k$. Erd\H os~\cite{MR144869}   proved that $n_1$ has   constant average;
we extend this to all $n_k$. Let $p_n $ be the $n$-th smallest prime
and $\pi(x)$ denote the number of primes $p\leq x $.
\begin{theorem}
\label{thm:main}
 Fix $ k\in \mathbb N$. Then 
\[
\lim_{x\to \infty } \frac{1}{\pi(x)} 
\sum_{\substack{\text{ prime } p \\p\leq x} }  %ES
n_k(p)  
 = \sum_{m\in  \mathbb N, m\geq k } \frac{p_m }{2^m} \frac{(m-1)!}{(m-k)!(k-1)!}
. \] 
 Furthermore, 
\[\lim_{k\to  \infty }  
\frac{ \sum_{ m\geq k } \frac{p_m }{2^m} \frac{(m-1)!}{(m-k)!(k-1)!}
}{p_{2k} }=1
.\]
\end{theorem}  
\begin{remark}
When combined the two statements   say that, on average,
 the $k$-th smallest prime quadratic non-residue  approximates the  $2k$-th  prime.
\end{remark}

Theorem~\ref{thm:main} will follow from a  more general statement (Theorem~\ref{thm:maingen}), which 
shows that for any function $f:\mathbb N^k \to \mathbb C$ satisfying only
a growth condition,
 $$f(n_1(p),   \ldots, n_k(p) )$$ has finite average. This allows us to  study the joint distribution  of $n_1(p)$ and $n_2(p)$ which is relevant 
to upcoming work of
Languasco and Moree
on quadratic residue bias of the divisor function.
  \begin{theorem}\label{thm:gap}
For every $z>1$ 
we have 
\[
\lim_{x\to\infty } \frac{
\#\{\text{prime } p\leq x: n_2(p) > z n_1(p)  \} 
}{\pi(x) } = \sum_{m= 1 }^\infty  2^{-n(m,z)} 
,\] where $n(m,z)$ is defined as the largest integer $n$ for which $p_n \leq z p_m$.
\end{theorem}
\begin{remark}\label{rem:cz}
The   inequality  $n(m,z)\geq m $ shows that  the sum     over $m$ in Theorem~\ref{thm:gap}
is rapidly convergent and it allows   a fast numerical approximation. For $z=3/2$ the first $13$ terms 
give the first $3$ correct digits. Precisely,
\[
\lim_{x\to\infty }
\frac{\#\left\{\text{prime }p\leq x : n_2(p)\leq \frac{3 }{2}n_1(p) \right\}}{\pi(x) } =0.350\ldots \ .
\]
\end{remark} The function $  M(p)=\min \{n_1(p), n_2(p)-n_1(p)\}$ plays a special r\^ole in the 
upcoming work of
Languasco and Moree.
We study its average and   largest value.

  \begin{theorem}\label{thm:miniaverg}
We have 
\[
\lim_{x\to\infty } \frac{1}{\pi(x) } 
\sum_{\substack{\text{ prime } p \\p\leq x} }  %ES
M(p) 
= 
\sum_{m=1}^\infty 
\sum_{k=m+1}^\infty 
\frac{\min\{p_m,p_k-p_m\}}{2^k}
=2.504\ldots 
\ .\]  
\end{theorem}
\begin{remark}
By definition one always has $M(p)\geq 2 $ and it is 
 somewhat  surprising that the average of $M(p)$ is so close to its minimum. 
This may   be explained by 
combining two facts: first that for the majority of primes, $M(p)$ equals $n_1$
and, secondly, that the average of $n_1$ is  smaller than the average of $n_2-n_1$.
To see the first point   we use  the case $z=2$ of Theorem~\ref{thm:gap} together with Remark~\ref{rem:cz}       to see that 
\[
M(p)= \begin{cases}
n_1(p) ,   & \text{with probability } 0.540\ldots \ , \\
n_2(p)-n_1(p),   & \text{with probability } 0.459\ldots \ .
 \end{cases}
\] For the second point, we use   the cases $k=1$ and $k=2$ of Theorem~\ref{thm:main}.
They show that the average of $n_1 $ and $n_2-n_1 $ is  
\[ 
\sum_{  m\geq  1 } \frac{p_m }{2^m}  = 3.674\ldots
\ \ 
\text{  and }
\ \ 
-1+\sum_{ m\geq 2 } \frac{p_m }{2^m}  (m-2) 
=4.352\ldots   
\]respectively.

\end{remark}
By Theorem~\ref{thm:main} with $k=1,2$ we know that $M(p)$ has finite average, hence, by Markov's inequality 
we see that for any fixed $c>0 $ one has 
\[
\frac{\#\{\text{prime }p\leq x : M(p)> c \log p \} }{\pi(x) } \to 0, 
\ \ \textrm{ as } x\to \infty
 .\]
In other words, $M(p)> c \log p $ with $0$ probability.
Our next result shows that there are infinitely many exceptions:
\begin{theorem}
\label{thm:linik} 
Fix any $0<c<1/10$. Then the inequality 
   $$ M(p)> c \log p $$ holds for  infinitely many primes $p$.
Conditionally on the Generalized Riemann Hypothesis there exists $c_0>0$
such that        $$ M(p)> c_0 (\log p) (\log \log p )$$ holds for  infinitely many primes $p$.
\end{theorem}To prove the conditional bound we follow the proof of Montgomery~\cite[Theorem 13.5]{MR0337847} quite closely; his proof 
regards the inequality    
$n_1(p)\geq c_0  (\log p) (\log \log p )$; our results includes his.

Although we shall not prove it, 
the logarithm lower bounds are tight under GRH: 
Ankeny~\cite{MR45159}
showed that $n_1(p)=O((\log p)^2 ) $ under GRH
and his proof can be modified     to show  $n_k(p)=O((\log p)^2 ) $ for all fixed $k$.
This would show that under GRH one has 
  $$M(p)= O((\log p)^2 ) .$$\textbf{Acknowledgements.} I would like to thank P. Moree for sending me his preprint
with A. Languasco 
on quadratic residue bias of the divisor function.
I would also like to thank him for 
 making me aware of the questions on the size of $n_k(p)$
without which this investigation would not have started.

\section{Preparatory lemmas}
\begin{lemma}
[Using Siegel--Walfisz's theorem]
\label{fixlem}
Assume that for each prime $p_k$ we are given $\epsilon_k \in \{1,-1\}$. 
Fix any constant $A>0$ and $n\in \mathbb N. $
Then for all $x\geq 2$ we have 
\[\#\left\{\textrm{prime } p\leq x: \forall 1\leq k \leq n \Rightarrow \left(\frac{p_k}{p}\right)=\epsilon_k  \right\} =\frac{\pi(x)}{2^n}  
+   O_A\left( 
\frac{   p_1 \cdots p_n x}{(\log x )^{ A}}
\right),\] where the implied constant depends at most on $A $.
\end{lemma}
\begin{proof} 
Let $q=8 \prod_{j=2}^ np_j$. We will show that there exists $\mathcal S\subset (\mathbb Z/q\mathbb Z)^*$ of cardinality $\phi(q) 2^{-n }$,
such that  a   prime $p>p_n$ satisfies the $\epsilon_k$-conditions 
if and only if 
$p\md q \in \mathcal{S}$.  
If $p\equiv 1 \md 4 $ then by quadratic reciprocity the conditions are equivalent to 
  \[
\left(\frac{2}{p} \right) =\epsilon_1,
1<i\leq n \Rightarrow 
\left(\frac{p}{p_i} \right)=\epsilon_i.\] Clearly the solubility for $p$ is periodic modulo $ q$.
  Furthermore, there   are exactly exactly $\prod_{j=2}^n  \frac{p_j-1}{2} $ such solutions $\md q $,
since for each odd prime $p_i$ exactly half elements of $\mathbb F_{p_i}^*$ are squares.
In the remaining case $p\equiv 3 \md 4 $ one ends up with the conditions 
  \[
\left(\frac{2}{p} \right) =\epsilon_1,
1<i\leq n \Rightarrow 
\left(\frac{p}{p_i} \right)=(-1)^{\frac{p_i-1}{2}}\epsilon_i
\] and the same considerations apply. The total number of solutions in $(\mathbb Z/q\mathbb Z)^*$ is 
$$ 2 \prod_{j=2}^n  \frac{p_j-1}{2} =\phi(q)2^{-n} .$$ 
Therefore, 
\begin{equation}
\label{eq:frgtyh}
\#\left\{\textrm{prime } p\leq x: \forall 1\leq k \leq n \Rightarrow \left(\frac{p_k}{p}\right)=\epsilon_k \right\}
=\sum_{t\in \mathcal S }
\#\{p\leq x : p\equiv t \md {q } \} .\end{equation}
Using the Siegel--Walfisz in the form~\cite[Eq. (5.77)]{MR2061214}
 we obtain 
\[\sum_{t\in \mathcal{ S}  }\left(\frac{\pi(x) }{\phi(q )}+O_A\left(\frac{x}{(\log x )^{ A}}\right)\right)\]
and the proof concludes by using $ \#\mathcal{S}=\phi(q)2^{-n} $.  \end{proof}

\begin{lemma}
[Using Brun--Titchmarsch's inequality]
\label{lem:brun}
Assume that for each prime $p_k$ we are given $\epsilon_k \in \{1,-1\}$. 
Then for all $x\geq 2$ and $n\in \mathbb N$ with $  p_1\ldots p_n \leq x^{9/10}$
  we have 
\[\#\left\{\textrm{prime } %ES
p\leq x: \forall 1\leq k \leq n \Rightarrow \left(\frac{p_k}{p}\right)=\epsilon_k \right\} =O\left(\frac{\pi(x)}{2^n}  \right)
,\] where the implied constant is absolute.
\end{lemma}
\begin{proof}  Let Let $q=8 \prod_{j=2}^ np_j$. 
The condition 
$  p_1\ldots p_n \leq x^{9/10}$
ensures that $q \leq x^{99/100}$, hence, injecting~\cite[Eq. (6.95)]{MR2061214} 
into~\eqref{eq:frgtyh}
yields  
\[
\ll
\sum_{t\in\mathcal S }\frac{\pi(x) }{\phi(q )},
 \] with an absolute implied constant.
The proof concludes by using that $ \#\mathcal S=\phi(q)2^{-n} $. 
  \end{proof}

\begin{lemma}
[Using Linnik's large sieve]
\label{lem:largesieve}
Let $  x  \geq z \geq 2  $ and $ q \in \mathbb N$.  
Then  the number of primes $p\leq x $ for which $ \left(\frac{\ell}{p}\right)=1$ for all primes $\ell$ except those dividing $q$
is 
\[ 
\ll
\frac{\#\{1\leq n \leq x^2: \gcd(n,q)=1 \} }{ \#\{1\leq n \leq x^2: \gcd(n,q)=1, p\mid n\Rightarrow p\leq z  \}}
 ,\] where the implied constant is absolute.
\end{lemma}
\begin{proof} 
This is a variation of the proof given in~\cite[Th. 7.16]{MR2061214}. We let 
\[\mathcal M   = \left\{1\leq n \leq x^2: \gcd(n,q)=1\right \} , 
\mathcal P   =\left \{p\leq x: \ell\leq z, \ell\nmid  q \Rightarrow \left(\frac{\ell}{p} \right)=1   \right\},
Q=x
,\] where $\ell $ denotes a prime.
Then the argument in~\cite[Th. 7.16]{MR2061214} works in our setting and yields the upper bound $\ll \#\mathcal M/\#\mathcal P$
for the number of $n\in \mathcal M$ that are quadratic  residues modulo every    $p\in \mathcal P$.
To conclude the proof note that if an integer $n\leq x ^2$ is coprime to $q$ and has all its prime divisors in the interval $[1,z]$,
then for every $p\in \mathcal P$ one has  $(\frac{n}{p})=1$ by the multiplicativity of the quadratic symbol.  
 \end{proof}

\begin{lemma}
[Smooth numbers; Hildebrand~\cite{MR831874}]
\label{lem:smooth}
Fix any   $\beta>1$ and $k\in \mathbb N$.
Then for all $x\geq 2 $ and 
all $q\in \mathbb N$ with at most $k$ distinct prime divisors    we have  
\[\#\left\{n\leq x : \gcd(n,q)=1, 
p\mid n \Rightarrow p\leq (\log x )^\beta \right\}
\gg 
x^{1-\frac{2}{\beta }  }, \] where the implied constant depends at most on    $\beta$   and $k$.
\end{lemma} 
\begin{proof} The special 
case $q=1$ is due to Hildebrand~\cite{MR831874} and one can   deduce from this the general  case as follows: 
  letting $Q$ be the square-free number composed of all prime divisors of $q$ that   are $\leq (\log x)^\beta$
we can write the quantity in the lemma as 
\[\sum_{d\mid Q}\mu(d) \#\{1\leq n \leq x/d:p\mid n \Rightarrow p\leq (\log x )^\beta  \}
. \] Note that $x/(\log x )^\beta < x/d\leq x $ since $d\leq Q\leq (\log x )^{k\beta}$. Hence, we obtain 
\[\sum_{d\mid Q}\mu(d) \left(\frac{x}{d}\right)^{1-1/\beta+o(1) }=x^{1-1/\beta+o(1) } \prod_{p\mid q, p\leq (\log x )^\beta} (1-p^{-1+1/\beta+o(1) }),\]
which is $ \geq x^{1-2/\beta} 2^{-k}
\gg_k x^{1-2/\beta}. $
This  is sufficient. \end{proof}

 \begin{lemma}[Bounding  $n_1(p)$; Burgess~\cite{MR93504}]
\label{lem:burgh}  
For any fixed   constant $c>\frac{1}{4\sqrt \mathrm e}$
we have   
$n_1(p)=O(p^c) $.\end{lemma} 

 \begin{lemma}[Burgess bound for $n_k(p)$; Banks--Guo~\cite{MR3606952}]
\label{lem:banksguo}  
For any fixed  $k\geq 2 $ and any fixed  constant $c>\frac{1}{4\sqrt \mathrm e}$
we have   
$n_k(p)=O_k(p^c) $.\end{lemma} 

\begin{lemma}
\label{lem:binom}
For any $k \in \mathbb N$ we have 
\[\sum_{n\in \mathbb N, n\geq k  } {n\choose k} 2^{-n}
=2.
\]Furthermore,   
 \[
\sum_{ n\in \mathbb N, n> 3 k   } n
{n\choose k} 2^{-n}\ll  \left(\frac{29}{32}\right)^{k}  \] with an absolute implied constant.
\end{lemma}
\begin{proof}
The first equation can be obtained by letting $x=1/2$ in 
$$
\frac{x^k}{(1-x)^{k+1}}
=\sum_{n=k}^\infty {n\choose k} x^n$$
that can be proved by differentiating $k+1$ times  the power series for $1/(1-x)$ around $x=0$.
For the second statement we note that if $n >3k$ then $n\geq 4  $, hence,  
 $$\frac{ (n+1) {n+1\choose k} 2^{-n-1} }{n{n\choose k} 2^{-n}} 
=
\frac{ 1+\frac{1}{n}    }{2   ( 1-\frac{k}{n+1}   )   } 
\leq 
\frac{ 1+\frac{1}{n}    }{2   ( 1-\frac{1}{3}   )   } 
=\frac{3}{4} \left(1+\frac{1}{n} \right)
\leq 
\frac{3}{4} \left(1+\frac{1}{4} \right)
=\frac{15}{16} .$$ 
By induction we can then obtain the following for $n> 3k$, 
$$
 n {n\choose k} 2^{-n}
\leq \frac{15}{16}   (n-1)  {n-1\choose k} 2^{-(n-1)}
 \leq \ldots 
\leq 
 \left(\frac{15}{16}\right)^{n-3k-1}   ( 3k+1)  {3k+1 \choose k} 2^{-3k-1}
.$$
By Stirling's approximation for the factorial we obtain 
$$  ( 3k+1)  {3k+1 \choose k} =
  ( 3k+1)  \frac{(3k+1)!}{k!(2k+1)!}
\ll  
 \frac{( 3k+1  )^{3k } \sqrt{k} }{
k^{k} 
 (2k+1  )^{2k }   
}=
 \frac
{3^{3k}  ( 1+1/3k  )^{3k } \sqrt{k} }
{
 2^{2k}  
(1+1/2k  )^{2k }   
} \ll 
 \frac
{3^{3k}   \sqrt{k} }
{  2^{2k}   } 
.$$ Hence for $n> 3k $ we have  
$$
 n {n\choose k} 2^{-n}
\ll 
 \left(\frac{15}{16}\right)^{n-3k-1} 
 \frac{3^{3k}   \sqrt{k} }{  2^{2k}   } 
2^{-3k}
\ll 
 \left(\frac{15}{16}\right)^{n-3k-1} 
 \left(\frac{29}{32}\right)^{k} 
$$ with an absolute implied constant.
Therefore,  
\[ \sum_{ n\in \mathbb N, n> 3 k   } n
{n\choose k} 2^{-n}
\ll
\sum_{  n\geq  3 k+1   } 
 \left(\frac{15}{16}\right)^{n-3k-1} 
 \left(\frac{29}{32}\right)^{k} 
\leq 
 \left(\frac{29}{32}\right)^{k} 
\sum_{  t=0   }^\infty  
 \left(\frac{15}{16}\right)^{t} 
\ll
 \left(\frac{29}{32}\right)^{k} 
.\] This is sufficient.  \end{proof}
\begin{lemma}
\label{lem:binom}
Fix  $k \in \mathbb N$, $c>0$ and assume that 
$f: \mathbb N^k\to \mathbb C$ satisfies
$$
 \max_{t_1,\ldots, t_k \leq x } |f(t_1,\ldots, t_k )|   =O\left(   x^c \right)  
 $$  for all $x\geq 1 $. Then 
\begin{equation}\label{vang1} 
\sum_ {\substack{ (m_1, m_2, \ldots, m_k) \in \mathbb N ^k \\ 1\leq m_1 <  \ldots < m_k\leq M }  } |f(p_{m_1} , \ldots, p_{m_k} ) |=
O((\log M)^c M^{c+k})
 \end{equation} and \begin{equation}\label{vang2} 
\sum_ {\substack{ (m_1, m_2, \ldots, m_k) \in \mathbb N ^k \\ 1\leq m_1 <  \ldots < m_k \\ m_k>M }  } 
\frac{|f(p_{m_1} , \ldots, p_{m_k} ) |}{2^{m_k} }
=O( (3/4)^M )
 \end{equation} hold for all $M> 1 $
with the implied constants depending at most on $c$ and $k$.  
 \end{lemma}
\begin{proof} First we note that $|f(p_{m_1} , \ldots, p_{m_k} ) |=O(p_{m_k}^c)$, hence, letting $m:=m_k$, we obtain the bound 
 \[\sum_ {\substack{ (m_1, m_2, \ldots, m_k) \in \mathbb N ^k \\ 1\leq m_1 <  \ldots < m_k\leq M }  } |f(p_{m_1} , \ldots, p_{m_k} ) |
\ll
\sum_{1\leq m  \leq M } p_{m}^c 
 \sum_ {\substack{ (m_1, m_2, \ldots, m_{k-1}) \in \mathbb N ^{k-1} \\ 1\leq m_1 <  \ldots < m_{k-1}<m  }  }1
\leq  \sum_{1\leq m  \leq M } p_{m}^c m^{k-1}
 .\] By the Prime Number Theorem we have $p_m \ll m \log m$, hence, this is $$\ll \sum_{m\leq M}m^{c+k-1} (\log m)^c\ll (\log M)^c M^{c+k}.$$
This proves the first assertion. To prove the second a similar argument yields
\[ 
\sum_ {\substack{ (m_1, m_2, \ldots, m_k) \in \mathbb N ^k \\ 1\leq m_1 <  \ldots < m_k \\ m_k>M }  } 
\frac{|f(p_{m_1} , \ldots, p_{m_k} ) |}{2^{m_k} }\ll
\sum_{m>M} \frac{m^{c+k-1}(\log m )^c}{2^m}
\ll
\sum_{m>M} \frac{m^{c+k}}{2^m}
.\] Using     $m^{c+k} =O((3/2)^m)$ shows that 
 the right-hand side is $O((3/4)^M)$.
\end{proof}

The next result is from~\cite[Theorem 2.1]{MR3086819}.
\begin{lemma}
[Linnik's constant]
 \label{lem:lin}
There exists a constant $C>0$ such that for every $q\in\mathbb N$ and $a\in \mathbb Z$ coprime to $q$ there exists a prime $p\leq C q^5$ satisfying 
$p\equiv a \md q $.
\end{lemma}  
 \section{The main theorem} 
Let $k\in \mathbb N$ and
assume that we are given any function $f:\mathbb N^k \to \mathbb C$ such that 
 \begin{equation}\label{cond:bnd}
 \max_{1\leq t_1,\ldots, t_k \leq x } |f(t_1,\ldots, t_k )|   =O\left(    x^{4\sqrt{\mathrm e} -\epsilon } \right)  
\end{equation}
for some constant $\epsilon>0$. The exponent $4\sqrt{\mathrm e}$ allows to control the size of $f$ at the first prime quadratic non-residues. 
Any improvement on the exponent in Lemma~\ref{lem:burgh}
 will allow to relax assumption~\eqref{cond:bnd} in what follows.
    \begin{theorem}
\label{thm:maingen} Assume that   $f:\mathbb N^k \to \mathbb C $ satisfies~\eqref{cond:bnd}.
We have \[
\lim_{x\to \infty } \frac{1}{\pi(x) } 
 \sum_{\substack{  \textrm{prime } p \\ p\leq x } }f(n_1(p), \ldots , n_k(p)  ) = 
\sum_{\substack{  (m_1,\ldots, m_k)  \in \mathbb N^k
\\ 1\leq m_1 <  \ldots < m_k } } 
\frac{f(p_{m_1} , \ldots, p_{m_k} )}{2^{m_k} } .\]
\end{theorem}
\begin{proof}
We write the sum over $p$
in the theorem as  
\[
\sum_{\substack{ (m_1,\ldots, m_k)  \in \mathbb N^k
\\ 1\leq m_1 <  \ldots < m_k } } 
f(p_{m_1} , \ldots, p_{m_k} )\#\{p\leq x: 1\leq i \leq k\Rightarrow n_i(p)=p_{m_i}  \}
\] and we split the sum according to the size of the largest prime as follows:
 \begin{enumerate}
\item $ p_{m_k} \leq A \log \log x $, 
\item $  A \log \log x <p_{m_k}  \leq \frac{\log x }{5}  ,$ 
\item $   \frac{\log x }{5}    <p_{m_k} \leq  (\log x )^{\beta},$ 
\item $ p_{m_k} >  (\log x )^{\beta}  $,
\end{enumerate}where $A $ and $\beta $ are   constants  that both strictly 
exceed $1$ and that will be specified later. 
Letting $n$ be the largest integer with $p_n \leq A \log \log x $ and using Lemma~\ref{fixlem}, 
the contribution of the first case is 
\[ \sum_ {1\leq m_1 <  \ldots < m_k\leq n   } f(p_{m_1} , \ldots, p_{m_k} ) 
  \left(\frac{\pi(x)}{2^{m_k} }  
+   O \left( 
\frac{  \prod_{1\leq j \leq m_k }p_j}{(\log x )^{  A\log 3  }}
\cdot
\frac{ x  }{(\log x )^{  2A }}
\right)\right)
.\]The prime number theorem shows that $p_1 \cdots p_t \leq 3^{p_t }$ holds for all large $t $, hence, 
 $$   \prod_{1\leq j \leq m_k }p_j  \leq  p_1 \cdots p_n   \leq 3^{ p_n } \leq 3^{ A \log \log x } = (\log x )^{A \log 3 }.$$
The contribution then becomes 
\begin{align*} 
&\sum_ {1\leq m_1 <  \ldots < m_k\leq n   } f(p_{m_1} , \ldots, p_{m_k} ) 
  \left(\frac{\pi(x)}{2^{m_k} }  
+   O \left( 
 \frac{ x  }{(\log x )^{  2A }}
\right)\right)
\\
=\pi(x)
& 
\sum_ {1\leq m_1 <  \ldots < m_k\leq n   }  
\frac{f(p_{m_1} , \ldots, p_{m_k} ) }{2^{m_k} } 
 +   O \left(
 \frac{ x  }{(\log x )^{  2A }}
 \sum_ {1\leq m_1 <  \ldots < m_k\leq n   } |f(p_{m_1} , \ldots, p_{m_k} ) |
\right) 
.\end{align*}
 By~\eqref{vang1} one gets 
   $$
\sum_ {1\leq m_1 <  \ldots < m_k\leq n   } |f(p_{m_1} , \ldots, p_{m_k} ) |
 \ll (\log n)^{4\sqrt \mathrm{e}} n^{4\sqrt \mathrm{e}+k} \ll n^{12+k}
  \ll (\log \log x )^{12+k } \ll (\log x)^A
.$$  
Therefore,  the first case contributes 
\begin{align*}
& \pi(x) 
\sum_ {1\leq m_1 <  \ldots < m_k\leq n   }  
\frac{f(p_{m_1} , \ldots, p_{m_k} ) }{2^{m_k} } 
 +   O \left(
 \frac{ x  }{(\log x )^{   A }}
 \right) 
\\=
&\left( 
\sum_ {1\leq m_1 <  \ldots < m_k \leq n   }  \frac{ f(p_{m_1} , \ldots, p_{m_k} ) 
}{2^{m_k}}  
\right) \pi(x)+o(\pi(x) )
 \end{align*} due to $\pi(x) \sim x/\log x$ and our assumption $A>1 $.  By~\eqref{vang2} the sum over $m_i$ converges absolutely, hence, 
the restriction $m_k \leq n $ can be removed at the cost of an admissible error.
In particular, the first case contributes 
   \[
\left( 
\sum_ {1\leq m_1 <  \ldots < m_k   }  \frac{ f(p_{m_1} , \ldots, p_{m_k} ) 
}{2^{m_k}}  
\right) \pi(x) (1+o(1) ) .\] 
It now  remains to show that all the other cases contribute $o(\pi(x) )$.

Let us next deal with Case $(2)$. If $p_t\leq \frac{\log x }{ 5} $ we obtain    $   p_1 \cdots p_t \leq 3^{p_t }  \leq x^{99/100 } $, hence, 
Lemma~\ref{lem:brun} yields  \begin{equation}\label{eq:brunti}
 \#\{p\leq x: 1\leq i \leq k\Rightarrow n_i(p)=p_{m_i}  \}
 \ll \frac{\pi(x) }{ 2^{m_k} }.\end{equation}
  Thus,   \begin{align*} &
\sum_{\substack{   1\leq m_1 <  \ldots < m_k\\ A \log \log x < p_{m_k} \leq \frac{ \log x }{5} } }   f(p_{m_1} , \ldots, p_{m_k} ) 
\#\{p\leq x: 1\leq i \leq k\Rightarrow n_i(p)=p_{m_i}  \}
\\
 &\ll 
 \pi(x) 
\sum_{\substack{  
1\leq m_1 <  \ldots < m_k\\ A \log \log x < p_{m_k} } } 
\frac{|f(p_{m_1} , \ldots, p_{m_k} ) | }{2^{ m_k}}
.\end{align*} This is $o(\pi(x))$ as the sum over $m_k$ is  the tail of an absolutely convergent series due to~\eqref{vang2}. 

Let us now move to Case $(3)$.
Every prime $\ell \leq \frac{\log x }{5}$ that is not in the set $\{p_{m_1}, p_{m_2}, \ldots, p_{m_{k-1}}\}$
must be a quadratic-residue modulo $p$ for each $p$ in Case $(3)$. Hence, 
denoting by $r$ the largest integer such that $p_r\leq \frac{\log x }{5}$, this means that 
 $$ j \in \{1,2\ldots, r\} \setminus \{m_1,m_2,\ldots, m_{k-1}\} \Rightarrow \left(\frac{p_j}{p}\right)=1 .$$Thus, applying Lemma~\ref{lem:brun}
 we obtain 
 \begin{align*} &
\sum_{\substack{   1\leq m_1 <  \ldots < m_k\\  \frac{ \log x }{5} < p_{m_k}  \leq (\log x )^\beta } }  | f(p_{m_1} , \ldots, p_{m_k} ) |
\#\{p\leq x: 1\leq i \leq k\Rightarrow n_i(p)=p_{m_i}  \}
 \\
\ll &\frac{\pi(x)}{2^{r-k} }
\sum_{\substack{   1\leq m_1 <  \ldots < m_k\\  \frac{ \log x }{5} < p_{m_k}  \leq (\log x )^\beta } }  | f(p_{m_1} , \ldots, p_{m_k} ) |
\ll \frac{\pi(x)}{2^r}\sum_{\substack{   1\leq m_1 <  \ldots < m_k\\    m_k   \leq (\log x )^\beta } }  | f(p_{m_1} , \ldots, p_{m_k} ) |
 ,\end{align*} with an  implied constant depending at most on $k$. By~\eqref{vang1} 
 the   sum over $ m_i  $  is 
\[ \ll (\log \log x )^{4\sqrt \mathrm{e}} (\log x) ^{\beta( 4\sqrt \mathrm{e}+k)}
\ll (\log x) ^{\beta( 7+k)}
  . \]  Using the fact that $p_r \sim r \log r  $
and $p_r \sim \frac{\log x }{5}$ we infer that $r\sim \frac{\log x }{5\log \log x } $.
In particular, $(\log x )^{\beta(7+k)} \leq 2^{r/2}$ for all large $x\geq 2 $, hence, 
 \[  \frac{\pi(x)}{2^{r }}  
(\log x )^{\beta(8+k)}\ll
 \frac{\pi(x)}{2^{r/2}}=o(\pi(x) ) .\] To deal with Case $(4)$ we use Lemmas~\ref{lem:burgh}-\ref{lem:banksguo}   as follows:
\begin{align*}
&\sum_{\substack{   1\leq m_1 <  \ldots < m_k\\    p_{m_k}  >  (\log x )^\beta } }  | f(p_{m_1} , \ldots, p_{m_k} ) |
\#\{p\leq x: 1\leq i \leq k\Rightarrow n_i(p)=p_{m_i}  \}
\\
\ll &
\max\left\{ |f(t_1,\ldots, t_k )| :1\leq t_1,\ldots,t_k   \leq x^{\frac{(1+ \epsilon)}{4\sqrt \mathrm e } } \right\}
\\
\times 
&
\sum_{ 1\leq m_1 <  \ldots < m_{k-1} \in \mathbb N  }
\#\{p\leq x: n_k(p)> (\log x )^\beta, n_i(p)=p_{m_i} \forall 1\leq i \leq k-1 \}
. \end{align*} 
By assumption~\eqref{cond:bnd} the maximum is $O(x^{1-\epsilon^2})$, whereas
  Lemma~\ref{lem:largesieve} with $q:=p_{m_1} \cdots p_{m_{k-1}}$ shows that 
 the overall 
contribution is 
\[\ll
\frac{x^{3-\epsilon^2}}{ \#\{n\leq x^2: \gcd(n,q)=1 , p\mid n\Rightarrow p\leq (\log z)^\beta \} } .  \]  
Alluding to Lemma~\ref{lem:smooth}
shows that this is $\ll 
x^{1-\epsilon^2+ 10/ \beta   } 
$.  Finally, the proof concludes by taking $  \beta=20/\epsilon^2 $.
\end{proof}
 
\section{Applications}
\subsection{Proof of Theorem~\ref{thm:main}}
Taking $f(t_1,\ldots, t_k )=t_k$ in Theorem~\ref{thm:maingen} proves that as $x\to \infty$, 
\[
\frac{1}{\pi(x) } \sum_{p\leq x} n_k(p)\
\to \mu_k:=
\sum_{n\geq k  } \frac{p_n}{2^n}
\#\{\b t \in \mathbb N^{k-1}: 1\leq t_1 < t_2< \ldots < t_{k-1} \leq n-1\}. \]
We can simplify this to 
 \[\mu_k 
=\sum_{n\geq k  } \frac{p_n}{2^n}{n-1\choose k-1} 
=k
\sum_{n\geq k  } \frac{p_n}{n}{n\choose k} 2^{-n}
. \] It now 
remains to 
  prove the asymptotic $\mu_k\sim p_{2k}$, which is equivalent to 
$$\lim_{k\to \infty} 
\frac{\mu_k }{    k \log k } = 2 
 $$ by the Prime Number Theorem. By the Prime Number Theorem     we have 
$p_n \sim n \log n \leq n^2 $ for all large $n$, hence, 
 $$
%ES \leq  
 \sum_{n>3k }\frac{p_n}{n(\log k ) }  {n\choose k} 2^{-n}
\leq   \sum_{n>3k }n {n\choose k} 2^{-n}\ll   \left(\frac{29}{32} \right)^k=o(1) 
 $$  by Lemma~\ref{lem:binom}.   We deduce 
\begin{equation}\label{eq:metafightfire}
 \frac{\mu_k}{k(\log k ) }
=
\sum_{n\in [k,3k]  } \frac{p_n}{n(\log k)}{n\choose k} 2^{-n}
+
o(1)
 .\end{equation}
Let us 
  fix an arbitrary 
 $\epsilon\in (0,1/2)$.
Since  $p_n \sim n \log n $ we know that 
for all sufficiently large $k $ and all $n\in [k,3k] $
one has $$ 1-\epsilon \leq 
\frac{p_n}  {n(\log k ) }\leq  1+\epsilon 
. $$ 
Therefore, 
\[
(1-\epsilon  )
 \sum_{n\in [k,3k]   } {n\choose k} 2^{-n}
\leq \frac{1}{\log k }
 \sum_{n\in [k,3k]  } \frac{p_n}{n}{n\choose k} 2^{-n}
\leq (1+\epsilon )
 \sum_{n\in [k,3k]   } {n\choose k} 2^{-n}.\]
Combining the two statements in Lemma~\ref{lem:binom}  
shows that  as $k\to\infty $ one has 
\[ \sum_{n\in [k,3k]   } {n\choose k} 2^{-n}=2 +o(1).\]
This implies that for all sufficiently large $k $ one has 
 \[
2(1-2\epsilon  )
 \leq \frac{1}{\log k }
 \sum_{n\in [k,3k]  } \frac{p_n}{n}{n\choose k} 2^{-n}
\leq 2(1+2\epsilon ) .\] 
By~\eqref{eq:metafightfire} this means that  
$  \mu_k\sim 2 k\log k$, which is sufficient.   

\subsection{Proof of Theorem~\ref{thm:gap}}
Taking $k=2$ and $f(t_1, t_2 )=\mathbf 1(zt_1<t_2 )$ in Theorem~\ref{thm:maingen} proves that as $x\to \infty$, 
\[
 \frac{
\#\{p\leq x: n_2(p) > z n_1(p)  \} 
}{\pi(x) } =\frac{1}{\pi(x)}
\sum_{p\leq x } f(n_1(p), n_2(p) ) \to 
\sum_{m\geq 1 }
\sum_{t\geq m+1 }
 \frac{\mathds 1 (zp_m<p_t)}{2^t}
.\] 
The sum over $t $ equals $\sum_{t\geq 1+n(m,z) } 2^{-t} =2^{-n(m,z)},$ where $n(m,z)$ is defined in the statement of Theorem~\ref{thm:gap}.

\subsection{Proof of Theorem~\ref{thm:miniaverg}}
Taking $k=2$ and $f(t_1, t_2 )=\min\{t_1,t_2-t_1\}$ in Theorem~\ref{thm:maingen} proves that as $x\to \infty$, 
\[
  \frac{1}{\pi(x) } 
\sum_{p\leq x } M(p) =\frac{1}{\pi(x)}
\sum_{p\leq x } f(n_1(p), n_2(p) ) 
\to 
\sum_{  1\leq m < t } \frac{\min\{p_m,p_t-p_m\}}{2^t}
  .\] 
 
\subsection{Proof of Theorem~\ref {thm:linik}}
We first prove the unconditional lower bound.
 For  $y\geq 2 $ we let $p_m $ be the largest prime with $p_m \leq   y/2 $ 
and $ p_n$ the largest prime with $p_n \leq y$. We let $q= 8 \prod_{j=2}^n p_j$
and define $$
\epsilon_m =-1,
j\in [1,n]\setminus \{m\} \Rightarrow \epsilon_j=1
.$$
As in the proof of Lemma~\ref{fixlem} there exists $t\md q$ with $\gcd(t,q)=1$ such that  
all primes $p\equiv t \md q $ satisfy  $(\frac{p_i}{p})=\epsilon_i$ for all $1\leq i \leq n $.
By Lemma~\ref{lem:lin} there exists such a prime in the range $p\leq C q^5$, hence, 
as $x\to\infty $ one has 
$$\log p \leq \log C+5 \log q =
\log C+5  \sum_{p\leq y}\log p  =5y  (1+o(1))  $$
 by the prime number theorem.
Furthermore, $n_2(p)>p_n$, thus, 
\[M(p)\geq \min\{p_m,p_n-p_m\}=  y/2   (1+o(1) ).\]
This shows that $M(p)\geq \frac{\log p }{10} (1+o(1))$, hence, for all fixed $c\in (0,1/10)$ the inequality 
 $M(p)\geq c \log p  $ holds infinitely often.

Let us now prove the unconditional lower bound. With $y, p_m,p_n$ and $\epsilon_m,\epsilon_j $ as above, we denote 
\[\mathcal P=\left\{p\in (x,2x]: 1\leq i \leq n \Rightarrow \left(\frac{p_i}{p}\right)=\epsilon_i \right\} .\]
Then, as in~\cite[pg. 128]{MR0337847}, we obtain  \begin{equation}\label{eq:montg}
 2^{\pi(y)-1}\sum_{p\in \mathcal P}  \log p   = \sum_{\substack{ k\in \mathbb N \\p'\mid k \Rightarrow p'\leq y  }}
(\psi(2x;\chi_k)-\psi(x;\chi_k)),\end{equation}
where $\psi(t;\chi):=
\sum_{1\leq p \leq x } \chi(p) $ and $\chi_k$ is a Dirichlet character   that is 
defined through 
$$\chi_k(p)
=\prod_{p'\mid k } \left(\frac{p_1}{p}\right)
.$$ We note here that $k$  divides $8\prod_{j=2}^n p_j$, hence, $\log k \ll \sum_{p\leq y} \log p \ll y $.
The contribution of $k=1$ gives $\sum_{x<p\leq 2x } (\log p ) = x+O(x^{1/2}\log^2 x) $ by the Riemann Hypothesis.
For the other terms the character $\chi_k$ is   non-principal, thus, GRH is known to imply that $\psi(x;\chi_k)\ll x^{1/2} \log^2 (kx)$,
with an absolute implied constant.
Hence, the right-hand side of~\eqref{eq:montg} equals
 \[  x+O(x^{1/2}\log^2 x)+O\left(\sum_{ k\mid 8 \prod_{j=2}^n p_j}
 x^{1/2} \log^2( k x ) \right)
= x+O\left(2^{\pi(y) } x^{1/2} ((\log x)^2+ y^2)\right) .
\] Denoting the implied constant by $c_1$, there exists  a small positive constant $c_0=c_0(c_1)$ such that 
if  $y=c_0(\log x)  (\log \log x ) $ then 
\[
c_1 2^{\pi(y) } x^{1/2} ((\log x)^2+ y^2)
\leq 
x/2
 \] by the prime number theorem $\pi(y) \sim y/\log y$. 
Thus, the left-hand side of~\eqref{eq:montg}
is strictly positive. This is sufficient.

\end{document}